\documentclass[10pt]{amsart}

\usepackage{amsthm,amsmath,amssymb,amscd}
\usepackage{amssymb}
\usepackage{graphics}

\usepackage{color}

\newtheorem{teor}{Theorem}[section]
 
\newtheorem{defin}[teor]{Definition}
\newtheorem{prop}[teor]{Proposition}

\newtheorem{lemma}[teor]{Lemma}

\newtheorem*{claim}{Claim}

\theoremstyle{definition}
\newtheorem{remar}[teor]{Remark}

\newcommand{\K}{K\"{a}hler}

\newcommand{\Ric}{\operatorname{Ric}}

\newcommand{\p}{\partial}

\begin{document}

\title{On the curvature of conic K\"ahler - Einstein metrics}
\author[Claudio Arezzo] {Claudio Arezzo}
\address{ arezzo@ictp.it }
\author{Alberto Della Vedova}
\address{alberto.dellavedova@unimib.it}
\author{Gabriele La Nave}
\address{lanave@illinois.edu}

\begin{abstract}
We prove a regularity result for Monge-Amp\`ere equations degenerate along smooth divisor on \K\ manifolds in Donaldson's 
spaces of $\beta$-weighted functions. We apply this result to study the curvature of \K\ metrics with conical singularities and give a geometric sufficient condition on the divisor for its boundedness.

\end{abstract}
\maketitle

\vspace{-,15in}


\section{Introduction}

Let $M$ be a compact complex manifold of dimension $n$ equipped with a \K\ form $\omega$. In this paper we study the curvature of \K\ metrics on $M\setminus D$ with cone singularities of cone angle $2\pi\beta$ transverse to 
a smooth divisor $D$, where $0<\beta<1$. This type of metrics were introduced by Tian \cite{tianlect} in the K\"ahler case 
 and have recently become
a key ingredient in the solution of the Tian-Yau-Donaldson conjecture about the existence of \K -Einstein metrics of positive curvature \cite{cds1, cds2, cds3, t}.

Metrics with this type of singularity can be produced by using as local potentials functions in the space $\mathcal D_w^{0,\alpha}$.
Such spaces were introduced by Donaldson \cite{Donaldson} and were denoted by $\mathcal C^{2,\alpha,\beta}$. In this paper we follow the notation of Jeffres, Mazzeo and Rubinstein \cite{jmr}.
A prototype for elements of $\mathcal D_w^{0,\alpha}$ is a function of the form $f + |s|^{2\beta}$, where $f$ is a smooth function on $M$ and, denoting by $(L,h)$ the holomorphic line bundle associated with $D$ equipped with a Hermitian metric, $|s|^2$ is the squared norm of a defining section for $D$.

An intriguing and important aspect of this theory is the behavior of the Riemann curvature tensor near the singularity.
This point has drawn the attention of many authors \cite{br,lr,sw,Donaldson}.
Besides its clear intrinsic interest, a great motivation for finding a \K\ conical metric with bounded curvature lies in the fact that once this is achieved, existence theorems for a large class of important PDEs become much easier and, for example, solving a conical version
of the Calabi Conjecture or \K -Einstein metrics can be done essentially relying on Yau's classical estimates, hence avoiding many of the
deep difficulties overcome in the general case by Jeffres, Mazzeo and Rubinstein \cite{jmr}.

It is then very natural to ask whether there exist geometric conditions which guarantee the boundedness of the Riemann tensor
at least for some conic metrics.
Our first result is the following

\begin{teor} \label{main1}

Let $(M,\omega)$ be a compact K\"ahler manifold, and let $(L,h)$ be the line bundle on $M$ associated to a smooth divisor $D$ equipped with a Hermitian metric.  Assume there exist $U \subset M$ and $V \subset L|_D$ neighborhoods of $D$ and a biholomorphism $\Upsilon : U \to V$ extending the identity on $D$.

Then for any $\alpha,\beta \in (0,1)$ with $\alpha<\beta^{-1}-1$ the following hold:
\begin{enumerate}
\item \label{exKbound}
in each \K\ class there exists a metric conic along $D$ of angle $2 \pi \beta$ with $\alpha$-H\"older continuous curvature,
\item \label{exCalabi}
given $\Omega \in c_1(M) - (1 - \beta) c_1(L)$ with local potentials in $\mathcal D_w^{0,\alpha}$, any K\"ahler metric $\tilde\omega$ conic along $D$ of angle $2 \pi \beta$ satisfying $\Ric(\tilde\omega) = \Omega +(1-\beta)\delta_D$ has $\alpha$-H\"older continuous curvature.
\end{enumerate}
\item
Moreover if $\lambda [\omega] = c_1(M) - (1 - \beta) c_1(L)$ for some real $\lambda$, then any \K -Einstein metric in $[\omega]$ conic along $D$ of angle $2 \pi \beta$ has $\alpha$-H\"older continuous curvature.

\end{teor}

General existence results related to point \eqref{exCalabi} of the theorem above and conic \K-Einstein metrics with $\lambda \leq 0$ follow by the work of Brendle \cite{br} for $\beta <1/2$, and  
Jeffres, Mazzeo and Rubinstein \cite{jmr} for general $\beta$.

Theorem \ref{main1} is in fact a corollary (using the mentioned fundamental works) of the following general regularity theorem for solutions
of complex Monge-Amp\`ere equations which is proved in Section \ref{sectCCM}. This result is likely to follow also from the theory 
developed in \cite{jmr} but we give an elementary proof which we believe is of independent interest:

\begin{teor}\label{boundedcurvatureintro}
Let $\alpha,\beta \in (0,1)$ with $\alpha<\beta^{-1}-1$, and let $\omega$ be a smooth K\"ahler metric on $M$.
Assume there exist a real number $k \geq 2\beta-1$ and non-negative smooth functions $u, F$ on $M$, with $u$ constant along $D$, such that
\begin{equation}
\left( \omega + i\partial \bar \partial u \right)^n = |s|^{2k} F \omega^n.
\end{equation}
Assume moreover that $\omega_0 = \omega + i \partial \bar \partial (u+|s|^{2\beta})$ is a conic K\"ahler metric of angle $2 \pi \beta$.

If $\omega_\phi= \omega_0 + i\partial \bar \partial \phi$ is a conic K\"ahler metric with $\phi \in \mathcal D_w^{0,\alpha}$ and it satisfies on $M \setminus D$ an equation of the form:
\begin{equation}\label{fakeMA2} \omega_\phi^n= e^f \omega_0^n \end{equation} 
with $f \in \mathcal D_w^{0,\alpha}$.
Then the norm of the Riemannian curvature $||Rm (\omega _{\phi})||_{\omega _\phi}$ is $\alpha$-H\"older continuous on $M$.

\end{teor}

Of course $D$ sits inside both $M$ and in the restriction of $L$ to $D$, which will be denoted from now on by $L|_D$.
In Section $3$ we then show that under the hypothesis on the existence of biholomorphisms of neighborhoods of $D$ (which in our paper will be referred to as the {\em{Grauert condition}})  in $M$ and in  $L|_D$, a conical background metric with all the required properties does exist:

\begin{teor}\label{degen}
Let $(M,\omega)$ be a $n$-dimensional compact K\"ahler manifold, and let $(L,h)$ be the line bundle on $M$ associated to a smooth divisor $D$ equipped with a Hermitian metric.  Assume there exist $U \subset M$ and $V \subset L|_D$ neighborhoods of $D$ and a biholomorphism $\Upsilon : U \to V$ extending the identity on $D$.
Then there exist a defining section $s \in H^0(M,L)$ of $D$ and a smooth real function $u$ on $M$ such that:
\renewcommand{\theenumi}{$(\roman{enumi})$}
\renewcommand{\labelenumi}{\theenumi}
\begin{enumerate}
\item \label{item::omegaD}
$u$ is constant along $D$, \\
\item \label{item::vanishingD}
the form $(\omega + i\partial\bar\partial u)^n$ vanishes exponentially along $D$, that is $|s|^{-2k} (\omega + i\partial\bar\partial u)^n$ is positive on $M \setminus D$ and extends to a smooth form on $M$ for any $k>0$. \\
\item \label{item::conic}
for any $\beta \in (0,1)$ the form
$\omega + i \partial \bar \partial (u + |s|^{2\beta})$ defines a conic K\"ahler metric of angle $2\pi\beta$.
\end{enumerate}
\end{teor}

The relevance of the reference metric $\omega + i \partial \bar \partial u$ rests on the exponential vanishing of its volume form along $D$. Indeed this is crucial for making the ratio $|s|^{2-2\beta} (\omega + i \partial \bar \partial (u+|s|^{2\beta}))^n/\omega^n$ of class $\mathcal D_w^{0,\alpha}$. In general this is not the case mainly for the presence of a summand of order $|s|^{2-2\beta}$ (note that the function $|s|^{2-2\beta}$ is not of class $\mathcal D_w^{0,\alpha}$ whenever $\beta>1/2$). 

The existence of smooth divisors satisfying the Grauert condition is a classical problem in complex geometry and clearly imposes a strong restriction and it is likely to be 
not the optimal one (in fact S. Donaldson in a private communication proposed the slightly weaker condition of {\em{splitting}} as the
optimal one to get boundedness of the curvature). At the end of the paper we briefly collect some of the known examples of 
Grauert divisors.

{\bf{Acknowledgements:}}
We wish to thank S. Donaldson for having shared some of his ideas on this problem with the first author.
The first author was partially supported by the FIRB Project ``Geometria Differenziale Complessa e Dinamica Olomorfa'',  while the second author was  supported by the Marie Curie IOF grant 255579 (CAMEGEST).

\section{Preliminaries on conic K\"ahler metrics}

Mainly to fix notation let us start by recalling some definitions.
First of all we consider some classes of functions which are relevant for defining and studing conic K\"ahler metrics.
Such classes have been introduced by Donaldson \cite{Donaldson} and have been extensively studied by Jeffres, Mazzeo and Rubinstein \cite{jmr}.

Fix $\beta \in (0,1)$ and let $K$ be the set of all positive integers less than $2 + 2 \beta$.
Given $k \in K$, consider on $\mathbf C^n$ the local change of coordinates $z=\psi_k(w)$, where
\begin{equation}
\psi_k (w_1,\dots, w_n) = (w_1,\dots , w_{n-1} , w_n^{1/\beta})
\end{equation}
is defined on the set $\Omega_k$ consisting of points $w$ satisfying
$\left| \arg(w_n) - \frac{k \beta \pi}{1+\beta} \right| < \frac{\beta \pi}{1+\beta}$
and $|w_n| >0$.
Note that $\psi_k$ is a biholomorphism onto its image, and that the images of $\psi_k$ as $k$ runs in $K$ cover $\mathbf C^n \setminus \{z_n = 0\}$.
Let $\alpha \in (0,1)$, and let $B \subset \mathbf C^n$ be a bounded open neighborhood of the origin.
A continuous function $f \in C^0(\overline B)$ is of class $C_w^{0,\alpha}$ if $\psi_k^*f$ is H\"older continuous of class $C^\alpha(\overline {B_k})$ for all $k \in K$, where $B_k = \psi_k^{-1}(\psi_k(\Omega_k) \cap B)$.
Moreover the function $f$ is of class $\mathcal D_w^{0,\alpha}$ if in addition for any $k \in K$, one has
\begin{equation}
\partial \bar \partial \psi_k^*f
= \sum_{i,j} f_{i j} dw_i \wedge d \bar w_j,
\end{equation}
where the functions $f_{ij}$ are in $C^\alpha(\overline B_k)$ and $f_{in},f_{nj}$ vanish at $w_n=0$. 
Note that $C_w^{0,\alpha}$ and $\mathcal D_w^{0,\alpha}$ are denoted by $\mathcal C^{,\alpha,\beta}$ and $\mathcal C^{2,\alpha,\beta}$ respectively by Donaldson \cite{Donaldson} (cf. also Rubsinstein \cite{yanir}).

\begin{remar}
The original definition of Donaldson's of the class $\mathcal C^{2,\alpha,\beta}$ involves a ~1-form $\epsilon$ which is only locally defined, precisely as the map $\psi_k$. In our notation defining $\epsilon$ via the equation $\psi_k^* \epsilon = \frac{\bar w_n}{|w_n|} d w_n$, the equivalence between the two definitions follows by observing that a H\"older continuous function which vanishes at $w_n=0$ stay such after multiplying by $\bar w_n / |w_n|$.
For convenience of the reader we include here a proof of this fact, which implies both inclusions of the two function spaces applied in one case as stated, and in the second case replacing $\frac{w_n}{|w_n|}$ with $\frac{\bar w_n}{|w_n|}$.
\end{remar}

\begin{lemma}
Let $\Omega \subset \mathbf C^n \times \mathbf C$ be some open neighborhood of the origin.
Let $f \in C^{\alpha}(\overline \Omega)$ for some $0<\alpha<1$ and assume that $f(z,0) = 0$ for all $(z,0) \in \Omega$.
Then the function $g$ defined by 
$$g(z,w) = \left\{ 
\begin{array}{ll}
f(z,w) \frac{w}{|w|} & \mbox{ if } w \neq 0 \\
0 & \mbox{ if } w=0
\end{array}\right.$$
belongs to $C^{\alpha}(\overline \Omega)$.
\end{lemma}
\begin{proof}
H\"older continuity of $f$ yields $\left| g(z',w) - g(z,w) \right| \leq C |z'-z|^\alpha$ for all $(z',w),(z,w) \in \Omega$, and $\left| g(z',w') - g(z,0) \right| \leq C |(z',w')-(z,0)|^\alpha$ for all $(z',w'),(z,0) \in \Omega$.
On the other hand, for all $(z,w),(z',w') \in \Omega$ with $w'\neq w$ and $w'w \neq 0$ one has
\begin{eqnarray*}
\left| g(z',w') - g(z,w) \right|
&=& \left| g(z',w') - f(z,w) \frac{w'}{|w'|} + f(z,w) \frac{w'}{|w'|} - g(z,w) \right| \\
&\leq& \left| g(z',w') - f(z,w) \frac{w'}{|w'|} \right| + \left| f(z,w) \frac{w'}{|w'|} - g(z,w) \right| \\
&\leq& \left| f(z',w') - f(z,w) \right| + |f(z,w)| \left| \frac{w'}{|w'|} - \frac{w}{|w|} \right|.
\end{eqnarray*}
Let $\varepsilon > 0$ and $(a,b)$ in the unit sphere of $\mathbf C^n \times \mathbf C$ such that $$(z',w') = (z,w) + \varepsilon (a,b).$$
Note that by assumptions $\varepsilon b$ and $w + \varepsilon b$ are nonzero. 
By H\"older continuity of $f$ one has
$$ \left| f(z',w') - f(z,w) \right| \leq C \varepsilon^\alpha. $$
For the same reason and the fact that $f(z,0)=0$ one also has
$$ |f(z,w)| \leq C |w|^\alpha. $$
Therefore
\begin{eqnarray*}
\left| g(z',w') - g(z,w) \right|
&\leq& C \varepsilon^\alpha + C |w|^\alpha \left| \frac{1 + \varepsilon b/w}{|1 + \varepsilon b/w|} - 1\right| \\
&=& C \varepsilon^\alpha + C \varepsilon^\alpha |b|^{\alpha} |\varepsilon b/w|^{-\alpha} \left| \frac{1 + \varepsilon b/w}{|1 + \varepsilon b/w|} - 1\right| \\
&\leq& C \varepsilon^\alpha \left( 1 + |\varepsilon b/w|^{-\alpha} \left| \frac{1 + \varepsilon b/w}{|1 + \varepsilon b/w|} - 1\right| \right).
\end{eqnarray*}
Now consider the function $\varphi$ on $\mathbf C \setminus \{0,1\}$ defined by
$$ \varphi (x) = |x-1|^{-2\alpha} \left| \frac{x}{|x|} -1  \right|^2,$$
so that one has
\begin{equation}\label{redtophi}
\left| g(z',w') - g(z,w) \right|
\leq C \varepsilon^\alpha \left( 1 + \sqrt{\varphi(1 + \varepsilon b/w)} \right).
\end{equation}
Since $\varepsilon = |(z',w') - (z,w)|$, H\"older continuity of $g$ will follow by showing that $\varphi$ is bounded.
To this end, passing to polar coordinates one has
\begin{eqnarray*}
\varphi(r e^{it})
&=& \frac{|e^{it}-1|^2}{|re^{it}-1|^{2\alpha}} \\
&=& \frac{2-2\cos t}{(r^2 - 2r \cos t + 1)^{\alpha}}
\end{eqnarray*}
whence it is clear that $\varphi(r e^{it}) \leq 4$ whenever $\cos t \leq 0$.
On the other hand, the lower bound $r^2 - 2r \cos t + 1 \geq 1 - (\cos t)^2$ together with the assumption $\cos t > 0$ give the bound $\varphi(r e^{it}) \leq 2 (1-\cos t)^{1-\alpha}/(1+\cos t)^{\alpha} \leq 2$.
Summarizing, we proved that $\varphi \leq 4$.
Substituting in \eqref{redtophi} then gives
\begin{equation}
\left| g(z',w') - g(z,w) \right|
\leq 3C |(z',w') - (z,w)|^\alpha,
\end{equation}
and the proof is complete.
\end{proof}

The definition of classes $C_w^{0,\alpha}$ and $\mathcal D_w^{0,\alpha}$ recalled above can be extended to the case of a pair $(M,D)$ where $M$ is a complex manifold and $D \subset M$ is a smooth divisor.
Here a continuous function $f$ on $M$ belongs to $C_w^{0,\alpha}$ if it is in $C^{\alpha}$ far from $D$, and for any choice of local complex coordinates $(z_1,\dots,z_n)$ centered at some point of $D$ such that $D$ is locally defined by $z_n=0$, the local expression of $f$ is in $C_w^{0,\alpha}$ according to definition above.
Analogously, $f$ is in $\mathcal D_w^{0,\alpha}$ if it is in $C^{2,\alpha}$ away from $D$, and locally around $D$ it is in $\mathcal D_w^{0,\alpha}$ in the sense defined above.

\begin{defin} 
A conic K\"ahler metric on $M$ of angle $2 \pi \beta$ along the smooth divisor $D$ is a K\"ahler current $\omega_c$
with local potential in $\mathcal D_w^{0,\alpha}$, and such that for any point $p\in D$, there are local coordinates $(z_1,\dots,z_n)$ centered at $p$ such that $D$ is locally defined by $z_n = 0$ and $\omega_c$ is uniformally equivalent to the model metric
$$ \omega_\beta = \sum_{k=1}^{n-1} i dz_k \wedge d\bar z_k + \beta^2 |z_n|^{2\beta-2} idz_n\wedge d\bar z_n. $$
\end{defin}

Now let $\omega$ be a smooth K\"ahler metric on $M$, and let $L$ be the line bundle associated to $D$ equipped with an arbitrary hermitian metric $h$.
As one can easily check, any defining section $s \in H^0(M,L)$ of $D$ can be rescaled so that $\omega + i \partial \bar \partial |s|^{2\beta}$ is a conic \K\ metrics of angle $2\pi\beta$ along $D$.

The interest in classes $C_w^{0,\alpha}$ and $\mathcal D_w^{0,\alpha}$ is motivated by regularity theory for the Poisson equation associated to a conic metric (\cite[Lemma 2.3]{jmr}).
The main case is $ \Delta_\beta v = f$, where
\begin{equation}
\Delta_\beta =\sum_{k=1}^{n-1} \frac{\partial^2}{\partial z_k \partial \bar z_k} + \beta^{-2} |z_n|^{2-2\beta} \frac{\partial^2}{\partial z_n \partial \bar z_n}
\end{equation}
is the Laplacian associated with the model metric $\omega_\beta$.
The following fundamental fact is due to Donaldson.

\begin{prop}\cite[Corollary 1 and Section 4.2]{Donaldson}  \label{weak-solution}
Let $B \subset \mathbf C^n$ be the unit ball centered at the origin, and let $D \subset B$ be the divisor defined by $z_n=0$.
Let $f$ be a function in $C_w^{0,\alpha}$ on $B$ for some $\alpha,\beta \in (0,1)$ such that $\alpha < \beta^{-1}-1$.
Then any weak solution $v$ of the Poisson equation $\Delta_\beta v = f$ is in $\mathcal D_w^{0,\alpha}$ on any ball $B'$ compactly contained in $B$.
\end{prop}

As hinted by Donaldson and proved by Jeffres, Mazzeo and Rubinstein \cite[Proposition 3.3]{jmr}
, the previous result extends to general conic K\"ahler metrics as follows, using the definition of conic (edge) laplacian as in equation $(14)$ of \cite{jmr}.

\begin{teor}\label{donaldson-regularity}
Let $B \subset \mathbf C^n$ be the unit ball centered at the origin, and let $D \subset B$ be the divisor defined by $z_n=0$.
Let $\omega_c$ be a conic K\"ahler metric on $B$ with local potential in $\mathcal D_w^{0,\alpha}$ for some $\alpha,\beta \in (0,1)$ such that $\alpha < \beta^{-1}-1$.
Let $f$ be a function in $C_w^{0,\alpha}$ on $B$, and let $v$ be a $W^{1,2}$ weak solution of the equation $\Delta_c v = f$.
Then $v$ is in $\mathcal D_w^{0,\alpha}$ on any ball $B'$ compactly contained in $B$.
\end{teor}

The behavior near the divisor of a function in $\mathcal D_w^{0,\alpha}$ can be conveniently described in terms of polyhomogeneous expansions as shown by Jeffres, Mazzeo and Rubinstein \cite{jmr}.
Although the next Proposition \ref{corolphexp} might be deduced from their theory, we believe it is of some interest to include here a direct proof relying on standard facts.

We remark that a key role in proving Proposition \ref{corolphexp} (which in turn will be used in the proof of Theorem \ref{boundedcurvatureintro}) is played by Lemma \ref{pointwise-secondderivative-estimate}, which extends to the case $\beta \in(0,1)$ an auxiliary result proved by Brendle \cite[Proposition A.1]{br} assuming $\beta < 1/2$.

\begin{prop}\label{corolphexp}
Fix real numbers $\alpha,\beta \in (0,1)$ such that $\alpha < \beta^{-1} - 1$.
Let $B \subset \mathbf C^n$ be the unit ball, and let $D \subset B$ be the divisor defined by $z_n=0$.
Let $f$ be a function of class $C_w^{0,\alpha}$ on B, and let $v$ be a bounded function of class $C^2$ on $B \setminus D$. 
Suppose that $v$ and $f$ satisfy 
\begin{equation}\label{maineqlocn}
\Delta_\beta v = f
\end{equation}
on $B \setminus D$.
Then there exist functions $a,b$ of class $C^\alpha$ on $B$, constant along D and such that the (complex valued) function $V = v - a|z_n|^{2\beta} - bz_n$ satisfies
$$ |z_n|^{2-2\beta} \left | \frac{\partial^2 V}{\partial z_n^2} \right|\leq C |z_n|^{\alpha' \beta},
\qquad |z_n|^{2-2\beta} \left | \frac{\partial^2 V}{\partial z_n^2} +\frac{1-\beta}{z_n} \frac{\partial V}{\partial z_n} \right | \leq C |z_n|^{\alpha' \beta}$$
for any $\alpha' < \alpha$ on  the complement of $D$ in the ball $|z|<1/6$.
\end{prop}
Note that $|z_n|^{2-2\beta} \left | \frac{\partial^2 V}{\partial z_n^2} +\frac{1-\beta}{z_n} \frac{\partial V}{\partial z_n} \right |$ is the absolute value of the derivative of $V$ with respect to $z_n^{\beta}$.
\begin{proof}

As already noticed the key point if the following
\begin{lemma}\label{pointwise-secondderivative-estimate}
Fix real numbers $\alpha,\beta \in (0,1)$ such that $\alpha < \beta^{-1}-1$. Let $\tilde{f}$ be a function of class $C^\alpha$ on the unit disc $B \subset \mathbf C$, and let $v$ be a bounded function of class $C^2$ on the punctured disc $B \setminus \{0\}$. 
Suppose that $v$ and $\tilde f$ satisfy 
\begin{equation}\label{maineqloc}
|z|^{2-2\beta} \frac{\partial^2 v}{\partial z \partial \bar{z}} = \tilde{f}(|z|^{\beta-1} z)
\end{equation}
on $B \setminus \{0\}$.
Then there exist $a,b \in \mathbf C$ such that for any $\alpha' < \alpha$ the function $V = v - a|z|^{2\beta} - bz$ satisfies
$$ |z|^{2-2\beta} \left | \frac{\partial^2 V}{\partial z^2} \right|\leq C |z|^{\alpha' \beta},
\qquad |z|^{2-2\beta} \left | \frac{\partial^2 V}{\partial z^2} +\frac{1-\beta}{z} \frac{\partial V}{\partial z} \right | \leq C |z|^{\alpha' \beta}$$
on the punctured disc $0<|z|<1/6$.
\end{lemma}
\begin{proof}
Let $a = \beta^{-2} \tilde f (0)$ and let
$$F(z) = v(z) - a |z|^{2\beta}$$
and 
$$h(z) = |z|^{2\beta-2} \left(\tilde f(|z|^{\beta-1} z) - \tilde f(0)\right).$$
Since $\tilde{f}$ is $\alpha$-H\"older continuous, there is a constant $C$ such that for all $\alpha'<\alpha$ one has
\begin{equation}\label{estimateh}
|h(z)| \leq C |z|^{2\beta-2+\alpha'\beta}.
\end{equation}

Substituting $F$ and $h$ in \eqref{maineqloc} gives
\begin{equation} \label{modified}
\frac{\partial^2 F}{\partial z \partial \bar{z}} = h.
\end{equation}
Since $F$ is $C^2$ on the punctured disc, one gets a bound on the first derivative of $F$ from Green's representation formula and a standard cutoff argument. More precisely, assuming $0<|z_0|<1/4$, one has
\begin{eqnarray}
\left| \frac{\partial F}{\partial z} (z_0) \right|
&\leq& C \int_{|z|<1/2} |z-z_0|^{-1} |h(z)| + C \nonumber \\
&\leq& C \int_{|z|<1/2} |z-z_0|^{-1} |z|^{2\beta-2+\alpha'\beta} + C \nonumber \\
&\leq& C |z_0|^{2\beta-1+\alpha'\beta} + C  \label{1stderest}
\end{eqnarray}

Thus $F$ turns out to be of class $W^{1,2}$ on $B_{1/4}$.
Moreover, it follows readily from \eqref{estimateh} that $h$ is of class $L^p$ for all $p< \bar p$, where $\bar p = 2/(2-2\beta - \alpha'\beta)$.
Therefore, $F$ is a weak solution of the Poisson equation \eqref{modified} and $L^p$ elliptic regularity \cite[Theorem 10.2.2, p. 276]{Jost} applies, so that $F$ is indeed of class $W^{2,p}$ on $B_{1/4}$ for all $p< \bar p$.

\begin{claim}
There exists $b \in \mathbf C$ such that
$|z_0|^{1-2\beta} \left | \frac{\partial F}{\partial z}(z_0) - b \right | \leq C  |z_0|^{\alpha'\beta}$ whenever $0 < |z_0| < 1/4$.
\end{claim}

First of all note that by the arbitrariness of $\alpha'$ one can suppose without loss of generality that $\alpha'\beta \neq 1-2\beta$.
We therefore have the following dichotomy: either $\alpha'\beta < 1 - 2\beta$ or $\alpha'\beta > 1 - 2\beta$.

If $\alpha'\beta < 1 - 2\beta$ then inequality \eqref{1stderest} holds with no additive constant, thus the claim is trivially verified with $b=0$.

On the othe hand, if $\alpha'\beta > 1 - 2\beta$ then $\bar p > 2$.
In this case by Sobolev embedding Theorem, the function $F$ turns out to be of class $C^{1,\gamma}$ on $\bar B_{1/4}$ for all $0 < \gamma < 1 - 2/\bar p$.
Let $b = \frac{\partial F}{\partial z} (0)$, so that by $\gamma$-H\"older continuity one has:
$$ \left | \frac{\partial F}{\partial z}(z_0) - b \right | \leq C  |z_0|^\gamma. $$
Substituting the expression of $\bar p$ in the bound for $\gamma$ gives that inequality above holds for any $\gamma < \alpha'\beta+2\beta-1$.
Thus the claim is proved thanks to arbitrariness of $\alpha'$ we started with.  
\medskip

Let $V=F-bz$ with $b$ given by the claim above. Clearly $V$ satisfies the Poisson equation
\begin{equation}\label{modifiedV}
\frac{\partial^2 V}{\partial z \bar \partial z} = h
\end{equation}
Let $0<|z| < 1/6$, and let $B'$ be the disc centered at $z$ with radius $|z|/2$.
Applying interior elliptic estimates to \eqref{modifiedV} on the disc $B'$ yields
\begin{equation} \label{2ndordest}
\left| \frac{\partial^2 V}{\partial z^2} \right| 
\leq C |z|^{-1} \sup_{B'} \left| \frac{\partial V}{\partial z} \right|
+ C \sup_{B'} |h| + C |z|^{\alpha'} [h]_{C^{\alpha'}(B')}. 
\end{equation}
Since $\tilde f$ is $\alpha'$-H\"older continuous, for all $x,y \in B'$ one has
\begin{eqnarray*} 
|h(x) - h(y)| 
&\leq& \left| |x|^{2\beta-2} (\tilde f(|x|^{\beta-1} x) - \tilde f(0)) - |y|^{2\beta-2} (\tilde f(|y|^{\beta-1} y) - \tilde f(0)) \right| \\
&\leq& C \left| |x|^{2\beta-2}-|y|^{2\beta-2}\right| |x|^{\alpha'\beta} 
+ C |y|^{2\beta-2} \left||x|^{\beta-1}x - |y|^{\beta-1}y\right|^{\alpha'} \\
&\leq& C |z|^{2\beta-2+\alpha'\beta}
\end{eqnarray*}

whence
$$ [h]_{C^{\alpha'}(B')}\leq C |z|^{2\beta-2+\alpha'(\beta-1)}.$$
Substituting this in \eqref{2ndordest} together with the information given by the claim above yields
$$|z_0|^{2-2\beta} \left | \frac{\partial^2 V}{\partial z^2} \right | \leq C |z|^{\alpha'\beta} \qquad \mbox{and} \qquad  |z|^{2-2\beta} \left | \frac{1}{z} \frac{\partial V}{\partial z} \right | \leq C |z|^{\alpha'\beta}$$
whence the thesis follows.
\end{proof}

Now we are in a position to conclude the proof of Proposition \ref{corolphexp}.
To this end, let $\tilde f$ be the continuous function on $B$ defined by $f(z) = \tilde f(z_1,\dots,z_{n-1},|z_n|^{\beta-1} z_n)$.
Since $f$ is in $C_w^{0,\alpha}$, the function $\tilde f$ is in $C^\alpha$. 
To see that, note that the pull-back $\psi_k^* f$ can be written as $\psi_k^* f(w) = \tilde f(w_1,\dots,w_{n-1}, |w_n|^{1-1/\beta} w_n^{1/\beta})$. Since $|w_n|^{1-1/\beta} w_n^{1/\beta}$ is a bi-Lipschitz function of $w_n$ and $\psi_k^* f(w)$ is in $C^\alpha$ by definition of $C_w^{0,\alpha}$, it follows that $\tilde f$ is of class $C^\alpha$, as claimed.

Denoted by $\Delta_e = \sum_{i=1}^{n-1} \frac{\partial^2}{\partial z_i \partial \bar z_i}$ the Euclidean laplacian on $\mathbf C^{n-1}$, equation \eqref{maineqlocn} can be written in the form
\begin{equation}\label{splitlaplacian}
|z_n|^{2-2\beta} \frac{\partial^2v}{\partial z_n \partial \bar z_n} = \beta^2 \left(f - \Delta_e v \right).
\end{equation}
By Proposition \ref{weak-solution} there exists a function $\tilde g$ of class $C^\alpha$ on $B$ 
such that r.h.s. of \eqref{splitlaplacian} is equal to $\tilde g (z_1,\dots,z_{n-1},|z_n|^{\beta-1}z_n)$.
Therefore, for any choice of $(z_1, \dots , z_{n-1})$, we are in the situation of Lemma \ref{pointwise-secondderivative-estimate}.
In particular, from its proof it follows that $a(z_1,\dots,z_{n-1}) = \beta^{-2} \tilde g (z_1,\dots,z_{n-1},0)$ and $b(z_1,\dots,z_{n-1}) = \frac{\partial (v-a|z_n|^{2\beta})}{\partial z_n}(z_1,\dots,z_{n-1},0)$ whenever $\beta \geq 1/2$, and they are zero otherwhise.
Therefore it remains to show that $b$ is of class $C^\alpha$ under the assumption $\beta \geq 1/2$.
To this end, let $\tilde v$ be defined by $v(z) = \tilde v(z_1,\dots,z_{n-1},|z_n|^{\beta-1}z_n)$ so that $b(z_1,\dots,z_{n-1}) = \varphi(z_1,\dots,z_{n-1},0)$ where $\varphi$ is defined by
\begin{multline} 
\varphi (z) = \frac{1+\beta}{2} |z_n|^{\beta-1}\frac{\partial (\tilde v-a|z_n|^2)}{\partial z_n}(z_1,\dots,z_{n-1},|z_n|^{\beta-1}z_n) 
\\
- \frac{1-\beta}{2} |z_n|^{\beta-3} \bar z_n^2 \frac{\partial (\tilde v-a|z_n|^2)}{\partial \bar z_n}(z_1,\dots,z_{n-1},|z_n|^{\beta-1}z_n).
\end{multline}
By Proposition \ref{weak-solution}, $\tilde v$ is of class $C^{1,\alpha}$ and we already proved that $a$ is of class $C^\alpha$.
Thus for all $(x,z_n),(y,z_n) \in B$ with $z_n \neq 0$ one has
\begin{eqnarray}
|\varphi(x,z_n) - \varphi(y,z_n)| \nonumber
&\leq& C |z_n|^{\beta-1} \left| \frac{\partial \tilde v}{\partial z_n} (x,|z_n|^{\beta-1}z_n) - \frac{\partial \tilde v}{\partial z_n} (y,|z_n|^{\beta-1}z_n)\right| \nonumber \\
&& +  C |z_n|^{\beta-1} \left| \frac{\partial \tilde v}{\partial \bar z_n} (x,|z_n|^{\beta-1}z_n) - \frac{\partial \tilde v}{\partial \bar z_n} (y,|z_n|^{\beta-1}z_n)\right| \nonumber \\
&& + C |z_n|^{2\beta-1} \left| a(x) - a(y) \right| \nonumber \\
&\leq& C |z_n|^{\beta - 1} \left| x - y \right|^\alpha. \label{holderphi}
\end{eqnarray}
Now suppose that $b$ is not of class $C^\alpha$. Then for any integer $N>0$ there exist $x_N,y_N \in B \cap D$ such that $|b(x_N) - b(x_N)| \geq N |x_N -y_N|^\alpha$. Therefore for any fixed $z_n$ suficiently small one has
\begin{eqnarray*}
|\varphi(x_N,z_n) - \varphi(y_N,z_n)|
&\geq& \left| |\varphi(x_N,0) - \varphi(y_N,0)| - |\varphi(x_N,z_n) - \varphi(x_N,0) - \varphi(y_N,z_n) + \varphi(y_N,0)| \right| \\
&\geq& C |b(x_N) - b(y_N)| \\
&\geq& C N |x_N - y_N|^\alpha,
\end{eqnarray*}
which contradicts \eqref{holderphi}. This proves that $b$ is $C^\alpha$ and we are done.
\end{proof}

\section{Curvature of conic metrics}\label{sectCCM}

Fix real numbers $\alpha,\beta \in (0,1)$ such that $\alpha < \beta^{-1} - 1$.
Let $\omega$ be a smooth K\"ahler metric on a complex manifold $M$, let $D$ be a smooth divisor on $M$, and let be fixed a Hermitian metric on the line bundle $L$ associated to $D$.
Given a defining section $s \in H^0(M,L)$ for $D$, assume there exist a real number $k \geq 2\beta-1$ and non-negative smooth functions $u, F$ on $M$, with $u$ constant along $D$, such that
\begin{equation}\label{eqvanishingvolume}
\left( \omega + i\partial \bar \partial u \right)^n = |s|^{2k} F \omega^n.
\end{equation}
Note that this condition is empty if $\beta \leq 1/2$, for in this case one can take $k=0$, $u=0$ and $F=1$.
Assume moreover that
\begin{equation}\label{defomega0}
\omega_0 = \omega + i \partial \bar \partial \left(u + |s|^{2\beta}\right)
\end{equation}
is a conic K\"ahler metric of angle $2\pi\beta$.

For any $\phi \in \mathcal D_w^{0,\alpha}$ such that $\omega_\phi = \omega_0 + i \partial \bar \partial \phi$ is a conic K\"ahler metric of angle $2\pi\beta$ consider the function $f$ defined by 
\begin{equation}\label{fakeMA.bis} 
\omega_\phi^n= e^{f} \omega_0^n.
\end{equation}

In this section we will show that if $f$ is in $\mathcal D_w^{0,\alpha}$ then the Riemannian curvature of $\omega_\phi$ has H\"older continuous norm, thus proving Theorem \ref{boundedcurvatureintro}.

Before proceeding further we need to introduce some notation. 
First of all let $(z_1,\dots,z_n)$ be holomorphic coordinates on $M$ such that $D$ has local equation $z_n=0$.
Consider the local change of coordinates $z=\psi(w)$, where
$$\psi (w_1,\cdots, w_n) = (w_1,\cdots w_{n-1} , w_n ^{\frac{1}{\beta}})$$
is defined on $\Omega \subset \mathbf C^n$ which is the intersection of an open neighborhood of the origin with the set consisting of points $w$ satisfying
$0 < \arg(w_n) < \frac{2\beta \pi}{1+\beta}$
and $|w_n| >0$.
Note that $\psi$ is a biholomorphism on its image. 
Thus $(w_1,\dots,w_n)$ are local coordinates defined on an open set $A \subset M \setminus D$ such that $\overline A \cap D \neq \emptyset$. 
This is important for we will be interested in the behavior of some functions around $D$.

Note that by definition of conic metric, the local expression 
\begin{equation*}
\omega_\phi = \sum_{\mu, \nu=1}^n g_{\mu \bar \nu} i dw_\mu \wedge d \bar w_\nu
\end{equation*} is uniformly equivalent to the standard euclidean metric in these coordinates.
Therefore H\"older continuity of the Riemannian curvature
\begin{equation}\label{curvature} 
Rm_{\mu\bar \nu \rho \bar \theta} = - \frac{\partial^2 g_{\mu \bar \nu}}{\partial w_\rho \partial \bar w_\theta} +
\sum_{\sigma,\tau=1}^n g^{\sigma \bar \tau }\frac{\partial g_{\mu \bar \tau}}{\partial w_\rho} \frac{\partial g_{\sigma \bar \nu}}{\partial \bar w_\theta}
\end{equation}
will follows directly from H\"older continuity of partial derivatives of coefficients $g_{\mu\bar\nu}$.
In particular Theorem \ref{boundedcurvatureintro} is a direct consequence of the following

\begin{prop}\label{mainestimateforbackground}
Fix real numbers $\alpha,\beta \in (0,1)$ such that $\alpha < \beta^{-1} - 1$. 
If $\phi, f \in \mathcal D_w^{0,\alpha}$ satisfy equation (\ref{fakeMA.bis}), then
$$\frac{\p g_{\mu \bar \nu}}{\p w_\gamma}, \frac{\p^2 g_{\mu \bar \nu}}{\p w_\gamma\p \bar w_{\delta}} \in  C^\alpha(\overline A)  $$
for any $\gamma, \delta  \in \{1,\cdots , n\}$.
\end{prop}

\begin{proof}
Let $h$ be a local potential for the background conic metric $\omega_0$, so that one has $\omega_0 = \sum_{\mu, \nu=1}^n \frac{\partial^2 h}{\partial w_\mu \partial \bar w_\nu} i dw_\mu \wedge d \bar w_\nu$, and any coefficient
\begin{equation}
g_{\mu \bar \nu} = \frac{\partial^2 (h+\phi)}{\partial w_\mu \partial \bar w_\nu}
\end{equation}
is of class $C^\alpha$ on $\overline A$.  
With these notations, taking the logarithm of \eqref{fakeMA.bis} gives
\begin{equation}\label{fakeMA.locbis}
\log \det \left( \frac{\partial^2 (h+\phi)}{\partial w_\mu \partial \bar w_\nu} \right)
= f + \log \det \left( \frac{\partial^2 h}{\partial w_\mu \partial \bar w_\nu} \right).
\end{equation}
Differentiating with respect to $w_\gamma$ yields
\begin{equation}\label{local-MA-2}
\Delta_\phi \left( \frac{\partial (h+\phi)}{\partial w_\gamma }\right) 
= \frac{\partial f}{\partial w_\gamma} + \Delta_0 \left(\frac{\partial h}{\partial w_\gamma}\right),
\end{equation}
where $\Delta_\phi$ and $\Delta_0$ denote the Laplacian with respect to the conic metrics $\omega_\phi$ and $\omega_0$ respectively.

\begin{claim}
$(\psi^{-1})^* \left(\Delta_0 \left(\frac{\partial h}{\partial w_\gamma}\right)\right)$ is of class $C_w^{0,\alpha}$.
\end{claim}

Before proving this Claim we conclude the argument that implies the statement of the Proposition.
Pulling back via $\psi^{-1}$ equation \eqref{local-MA-2} to $z_i$'s coordinates, one gets a conic Poisson equation with r.h.s. of class $C_w^{0,\alpha}$.
Moreover, thanks to Proposition \ref{corolphexp}, it is immediate to check that $\frac{\partial(h + \phi)}{\partial w_\gamma}$ is indeed a $W^{1,2}$ solution of this Poisson equation.

Thus by Theorem \ref{donaldson-regularity} one has $(\psi^{-1})^* \left( \frac{ \p (h+\phi)}{\p w_\gamma }\right) \in \mathcal D_w^{0,\alpha}$, whence by definition of class $\mathcal D_w^{0,\alpha}$ it follows that the complex hessian of $\frac{ \p (h+\phi)}{\p w_\gamma }$ is in $C^\alpha (\overline A)$. That is
\begin{equation*}
\frac{\p g_{\mu \bar \nu}}{\p w_\gamma} \in C^\alpha(\overline A) 
\end{equation*}
for all $w_\mu$ and $\bar w_\nu$.
Differentiating equation \eqref{local-MA-2} with respect to $\bar w_\delta$ yields
\begin{equation}
\Delta_\phi \left( \frac{\partial^2 (h+\phi)}{\partial \bar w_\delta \partial w_\gamma} \right)
= \frac{\partial}{\partial \bar w_\delta} \left( \frac{\partial f}{\partial w_\gamma} + \Delta_0 \left(\frac{\partial h}{\partial w_\gamma}\right) \right)
- \sum_{i,j,h,k=1}^n g^{i \bar k} g^{h \bar j} \frac{\partial g_{h \bar k}}{\partial \bar w_\delta} \frac{\partial g_{i \bar j}}{\p w_\gamma}.
\end{equation}
\begin{claim}
$(\psi^{-1})^*  \frac{\partial}{\partial \bar w_\delta} \left( \Delta_0 \left(\frac{\partial h}{\partial w_\gamma}\right) \right)$ is of class $C_w^{0,\alpha}$.
\end{claim}

As we did earlier, we postpone the proof of this Claim.
Arguing as above one concludes that the complex hessian of $\frac{\partial^2 (h+\phi)}{\partial \bar w_\delta \partial w_\gamma}$ is in $C^\alpha (\overline A)$.
Therefore differentiating by $w_\mu$ and $\bar w_\nu$ gives the statement.

\medskip
Finally it remains to prove the Claims.
Let $i\Theta$ be the curvature of the Chern connection of the Hermitian line bundle $L$.
Thanks to the identity 
\begin{equation}
i \partial \bar \partial |s|^{2\beta} = -\beta |s|^{2\beta} \Theta + \beta^ 2|s|^{2\beta - 4} i \partial |s|^2 \wedge \bar \partial |s|^2, 
\end{equation}
starting from \eqref{defomega0}
one calculates:
\begin{equation*}
\omega_0^n 
= \left( \omega + i \partial \bar \partial u - \beta |s|^{2\beta} \Theta + n\beta^2 |s|^{2\beta-4} i \partial |s|^2 \wedge \bar \partial |s|^2 \right) \wedge \left( \omega + i \partial \bar \partial u - \beta |s|^{2\beta} \Theta \right)^{n-1}.
\end{equation*}
Clearly this defines smooth functions $a_j$ depending on $\Theta$, $\omega+i\partial\bar\partial u$, and $|s|^{-2} i \partial |s|^2 \wedge \bar \partial |s|^2$ such that one has the expansion
\begin{equation}\label{eqexpansionvolume}
\omega_0^n = (\omega + i \partial \bar \partial u)^n + |s|^{2\beta-2} \sum_{j=0}^{n-1} a_j |s|^{2j\beta} \omega^n.
\end{equation}
Since $u$ is assumed to be constant along $D$, one has
$a_0 = \frac{n\beta^2 |s|^{- 2} i \partial |s|^2 \wedge \bar \partial |s|^2 \wedge (\omega + i\partial\bar\partial u)^{n-1}}{\omega^n} >0 $. 
On the other hand, plugging \eqref{eqvanishingvolume} into \eqref{eqexpansionvolume} gives
\begin{equation}\label{eqratiovol}
\frac{|s|^{2-2\beta}\omega_0^n}{\omega^n}
= |s|^{2(k+1-\beta)}F + \sum_{j=0}^{n-1} a_j |s|^{2j\beta}.
\end{equation}
The hypothesis on $k$ implies $k+1-\beta \geq \beta$, whence it follows that the r.h.s. of \eqref{eqratiovol} is of class $\mathcal D_w^{0,\alpha}$.
Since it is positive as well, one concludes that
\begin{equation}
\label{cazzucazzu}
\log\left(\frac{|s|^{2-2\beta}\omega_0^n}{\omega^n}\right) \in \mathcal D_w^{0,\alpha}.
\end{equation}
Therefore there exists a function $H$ of class $\mathcal D_w^{0,\alpha}$ such that in local coordinates one has
\begin{equation}
\log \det \left( \frac{\partial^2 h}{\partial w_\mu \partial \bar w_\nu } \right) = \psi^* H.
\end{equation}

Taking derivatives with respect to $w_\gamma$ and $\bar w_\delta$ then proves the claims thanks to definition of classes $C_w^{0,\alpha}$ and $\mathcal D_w^{0,\alpha}$.
\end{proof}

In order to prove Theorem \ref{main1} we need the following
 
\begin{prop}\label{C2abricci}
Given any $\Omega \in  c_1(M) - (1-\beta) c_1(L)$ with $\mathcal D_w^{0,\alpha}$ local potential, there exists $F^\Omega \in \mathcal D_w^{0,\alpha}$ such that
\begin{equation}\label{riccipot}
\Ric(\omega_0) = \Omega + (1-\beta) \delta_D + i\partial \bar \partial F^\Omega.
\end{equation}

Moreover if $ c_1(M) = \lambda [\omega] + (1-\beta) c_1(L)$ for some real $\lambda$, then $\omega_0$ has Ricci potential of class $\mathcal D_w^{0,\alpha}$, i.e. there exists $F^\lambda \in \mathcal D_w^{0,\alpha}$ such that 
\begin{equation}\label{Riccipotential}
\Ric (\omega_0) = \lambda \omega_0 +(1-\beta)\delta_D + i \partial\bar{\partial}F^\lambda.
\end{equation}
\end{prop}
\begin{proof}
The curvature $\Theta$ of the Hermitian line bundle $L$ and the integration current $\delta_D$ alog $D$ are related by the Lelong equation $ \Theta + i \partial \bar \partial \log |s|^2 = \delta_D$.
Therefore by the $\partial \bar \partial$-lemma there is a function $f^{\Omega}$ on $M$ of class $\mathcal D_w^{0,\alpha}$ such that
$$ \Ric(\omega) = \Omega + (1-\beta) \Theta + i \partial \bar \partial f^{\Omega}.$$
Moreover we have
$$ \Ric(\omega_0) = \Ric(\omega) -i\partial\bar \partial
\log \frac{\omega_0^n}{\omega^n} = \Omega + (1-\beta) \delta_D +
i \partial \bar \partial f^{\Omega} -i\partial\bar \partial
\log \frac{|s|^{2-2\beta}\omega_0^n}{\omega^n} $$
which yields \eqref{riccipot} with $F^\Omega$ given by
\begin{equation}
\label{fbeta} F^\Omega = f^\Omega - \log \frac{|s|^{2-2\beta} \omega_0^n}{\omega^n}.
\end{equation}

As we just observed in equation (\ref{cazzucazzu}) in the proof of Proposition \ref{mainestimateforbackground}, the function $\log\left(\frac{|s|^{2-2\beta}\omega_0^n}{\omega^n}\right)$ is of class $D_w^{0,\alpha}$, and hence $F^\Omega$ belongs to the same space.

For the second part, by $\partial \bar \partial$-lemma there is a smooth function $f_0$ on $M$ such that
$$ \Ric(\omega) = \lambda \omega + (1-\beta) \Theta + i \partial \bar \partial f_0.$$
Again by the Lelong equation $ \Theta + i \partial \bar \partial \log |s|^2 = \delta_D$, one then has
$$ \Ric(\omega) = \lambda \omega_0 + (1-\beta) \delta_D + i \partial \bar \partial \left(f_0 - \lambda(u+|s|^{2\beta}) - \log |s|^{2-2\beta}\right). $$
Thus \eqref{Riccipotential} holds by
$$ \Ric(\omega_0) 
= \Ric(\omega) - i\partial\bar \partial \log \frac{\omega_0^n}{\omega^n} $$
with
\begin{equation}\label{eq::exprf}
F^\lambda = f_0 - \lambda(u+|s|^{2\beta}) - \log \frac{|s|^{2-2\beta}\omega_0^n}{\omega^n}.
\end{equation}
\end{proof}

\section{The Background Metric}

We now face the problem of whether there exists a conic metric $\omega_0$ of a given angle $2\pi\beta$ of the form required by the hypotheses of Theorem \ref{boundedcurvatureintro}.
Let us first start with an elementary but useful observation.

\begin{lemma}\label{Meta}
Let $\theta$ be a smooth non-negative function on $\mathbf R$ with support in $[-1,1]$ such that $\int_{\mathbf R} \theta(x)dx = 1$ and $\int_{\mathbf R} x \theta(x) dx =0$.
For any $\eta>0$ the function $M_\eta : \mathbf R^2 \to \mathbf R$ defined by
\begin{equation}\label{Metashift}
M_\eta (t) = \int_{\mathbb R^2} \max(t_1 + h_1, t_2 + h_2) \theta\left( \eta h_1 \right) \theta\left( \eta^{-1} h_2 \right) dh_1 dh_2
\end{equation}
is smooth, convex and satisfies
\begin{equation}\label{eq::Metasymp}
M_\eta (t) =
\left\{ 
\begin{array}{ll} 
t_1 & \mbox{ if }\quad t_1 \geq t_2 + \eta + \eta^{-1} \\ 
t_2 & \mbox{ if }\quad t_2 \geq t_1 + \eta + \eta^{-1}.
\end{array} 
\right. 
\end{equation}
Moreover one has
$$ 0 \leq \frac{\partial M_\eta (t)}{\partial t_i} \leq 1. $$
\end{lemma}
\begin{proof}
This is  a special case of \cite[page 43, Lemma 5.18(a)]{DemaillyCAG}.
\end{proof}

We can now proceed with the proofs of Theorems \ref{degen} and \ref{main1}.
\begin{proof}[Proof of Theorem \ref{degen}]
Let $p$ be the bundle projection from $L|_D$ to $D$ and let $\iota$ be the inclusion of $D$ into $L|_D$ as the zero section.
Since $p$ is holomorphic, $p^*$ maps the Dolbeault cohomology groups $H^{1,1}(D)$ into $H^{1,1}(L|_D)$, and being a retraction 
$p^* : H^{1,1}(D) \to  H^{1,1}(L|_D)$ and $\iota^* : H^{1,1}(L|_D) \to  H^{1,1}(D)$ are isomorphisms inverse to each other. 
As a consequence the restriction to $V$ of $p^*\omega|_D$ is cohomologous to $(\Upsilon^{-1})^* \omega|_U$ so that there is a smooth function $v$ on $V$ such that
\begin{equation}
(\Upsilon^{-1})^* \omega|_U = p^*\omega|_D + i \partial \bar \partial v.
\end{equation}
Since $\Upsilon$ induces the identity on $D$, the function $v$ is constant along $D$, and one can assume that $v$ vanishes on $D$.

Now consider the smooth function on $L|_D$ which associates $h(\ell,\ell)$ to any $\ell \in L|_D$, being $h$ the given Hermitian metric on $L$. By a little abuse of notation we will denote such function by $h$.
One has
\begin{eqnarray}
\nonumber \left( p^*\omega|_D + i \partial \bar \partial e^{-\frac{1}{h}} \right)^n 
&=&
\left( p^*\omega|_D + \left( \frac{e^{-\frac{1}{h}}}{h^4} i \partial h \wedge \bar \partial h - \frac{2e^{-\frac{1}{h}}}{h^3} i \partial h \wedge \bar \partial h + \frac{e^{-\frac{1}{h}}}{h^2} i \partial \bar \partial h \right) \right)^n \\
&=&
\frac{e^{-\frac{1}{h}}}{h^4}\left( n \, p^*\omega|_D^{n-1} \wedge i \partial h \wedge \bar \partial h + R \right)
\end{eqnarray}
for some smooth function $R = O\left(h^2\right)$ as $h \to 0$.
Since $i\partial h \wedge \bar \partial h$ is smooth, non-negative and positive in the direction of fibers of $L|_D$, after possibly shrinking $V$ it follows that
\begin{equation}\label{eq::posU-D}
\left( p^*\omega|_D + i \partial \bar \partial e^{-\frac{1}{h}} \right)^n  >0 \quad \mbox{on} \quad  V \setminus D.
\end{equation}
Let $s \in H^0(M,L)$ be a defining section of $D$.
One can readily check by local calculations that $\Upsilon^*h = |s|^2 \tilde h$ for some positive smooth function $\tilde h$ on $U$, whence
\begin{equation}
\Upsilon^*\left( p^*\omega|_D + i \partial \bar \partial e^{-\frac{1}{h}} \right)^n
=
\frac{e^{-\frac{1}{|s|^2 \tilde h}}}{|s|^8 \tilde h^4} \Upsilon^*\left( n \, p^*\omega|_D^{n-1} \wedge i \partial h \wedge \bar \partial h + R \right).
\end{equation}
Thus by \eqref{eq::posU-D} for any $k>0$ letting
$$ |s|^{-2k} \Upsilon^*\left( p^*\omega|_D + i \partial \bar \partial e^{-\frac{1}{h}} \right)^n $$
defines a smooth $2n$-form on $U$ that is positive on $U\setminus D$.

Letting
$$ \tilde u = \Upsilon^*(e^{-\frac{1}{h}} - v)$$
defines a smooth function on $U$ satisfying ${\omega_{\tilde u}}|_D= \omega|_D $, and for any $k>0$ the form $|s|^{-2k} \omega_{\tilde u}^n$ is positive on $U \setminus D$ and extends to a smooth form on $U$.

The function $u$ will be finally obtained by gluing $\tilde u$ together with a suitable function $q$ which is $\omega$-plurisubharmonic away from $D$. Clearly \ref{item::omegaD} and \ref{item::vanishingD} will hold for such a function $u$ since these conditions are local around $D$ and we proved above that they are satisfied by $\tilde u$ on $U$.
Let
$$ q = \frac{1}{\eta^2} + \eta \log |s|^2,$$
where $\eta > 0$ is a constant to be determined. Note that $\omega + i \partial \bar \partial q >0$ on $M \setminus D$ for all sufficiently small $\eta$. Choose $r >0$ so that $U_r = \{ |s|^2<r\} \subset U$.
Perhaps taking $\eta$ smaller, one can make $q$ satisfying 
$$q>\tilde u+\eta+\eta^{-1} \mbox { on } \partial U_r.$$
On the other hand, since $q$ goes to $-\infty$ around $D$ and $\tilde u$ is bounded, one can find $r'<r$ such that $\tilde u > q+\eta+\eta^{-1}$ on $\overline{U_{r'}}$. 
Now consider $M_\eta$ as in Lemma \ref{Meta} and define $u$ by
$$
u(x) = 
\left\{
\begin{array}{ll} 
\tilde u(x) & x \in U_{r'} \\
M_\eta(\tilde u(x),q(x)) & x \in U_r \setminus U_{r'}\\ 
q(x) & x \in M \setminus U_r
\end{array}
\right.
$$
Note that we chose $r,r'$ so that $u=q$ around $\partial U_r$ and $u=\tilde u$ around $\partial U_{r'}$ by Lemma \ref{Meta}. This fact together with smoothness of $M_\eta$ implies that $u$ is smooth.
By definition of $u$, we have to check the $\omega$-plurisubharmonicity only on the set $U_r \setminus U_{r'}$.
From the convexity of $M_\eta$ one gets the following lower bound:
\begin{eqnarray*}
i \partial \bar \partial u &=& 
i \partial \bar \partial M_\eta (\tilde u,q) \\
&=& i \partial \left( \frac{\partial M_\eta}{\partial t_1}(\tilde u,q) \bar \partial \tilde u + \frac{\partial M_\eta}{\partial t_2}(\tilde u,q) \bar \partial q \right) \\
&=& \frac{\partial^2 M_\eta}{\partial t_1^2}(\tilde u,q) i \partial \tilde u \wedge \bar \partial \tilde u + \frac{\partial^2 M_\eta}{\partial t_1 \partial t_2}(\tilde u,q) i \partial \tilde u \wedge \bar \partial q \\ 
 && + \frac{\partial^2 M_\eta}{\partial t_2 \partial t_1}(\tilde u,q) i \partial q \wedge \bar \partial \tilde u + \frac{\partial^2 M_\eta}{\partial t_2^2}(\tilde u,q) i \partial p \wedge \bar \partial q\\
 && + \frac{\partial M_\eta}{\partial t_1}(\tilde u,q) i \partial \bar \partial \tilde u + \frac{\partial M_\eta}{\partial t_2}(\tilde u,q) i \partial \bar \partial q \\
 &\geq& \frac{\partial M_\eta}{\partial t_1}(\tilde u,q) i \partial \bar \partial \tilde u + \frac{\partial M_\eta}{\partial t_2}(\tilde u,q) i \partial \bar \partial q.
 \end{eqnarray*}
On the other hand, by definition of $q$ it follows that
$$i \partial \bar \partial q = - \eta \Theta, $$
where $i\Theta$ the curvature of the Chern connection of the Hermitian line bundle $(L,h)$, 
whence
$$ \omega + i \partial \bar \partial u \geq \omega + \frac{\partial M_\eta}{\partial t_1}(\tilde u,q) i \partial \bar \partial \tilde  u - \eta \frac{\partial M_\eta}{\partial t_2}(\tilde u,q) \Theta. $$
By $\omega$-plurisubharmonicity of $\tilde u$ and Lemma \ref{Meta} it is clear that $\eta$ can be chosen small so that $\omega + i \partial \bar \partial u>0$.

Now it remains to check \ref{item::conic}. On $M \setminus D$ one has
\begin{equation*}
i \partial \bar \partial |s|^{2\beta} 
= -\beta |s|^{2\beta} \Theta + \beta^ 2|s|^{2\beta} i \partial \log |s|^2 \wedge \bar \partial \log|s|^2
\geq - \beta |s|^{2\beta} \Theta,
\end{equation*}
whence is clear that $s$ can be rescaled so that $\omega + i \partial \bar \partial (u + |s|^{2\beta}) >0$ for all $\beta \in (0,1)$.
Indeed one can readily check that assuming $\sup_M |s|^2 < 1$ one has the bound $\beta|s|^{2\beta} \leq (-e \log \sup_M |s|^2)^{-1}$ which does not depend on $\beta$.
This concludes the proof.
\end{proof}

\begin{proof}[Proof of Theorem \ref{main1}] By Theorem \ref{degen} there exists a conic background metric in $[\omega]$ with local potential and Ricci potential of class $\mathcal D_w^{0,\alpha}$, thus by Theorem \ref{boundedcurvatureintro} such a metric has $\alpha$-H\"older continuous curvature. Then arbitrariness of $\omega$ proves point \eqref{exKbound}. In the same way point \eqref{exCalabi} and the statement about \K-Einstein metrics follow by Theorem \ref{boundedcurvatureintro} and Proposition \ref{C2abricci}.
 \end{proof}

\noindent We can then ask when a divisor satisfies the Grauert condition, i.e. the property of having tubular neighborhoods $U \subset M$ and $V \subset L|_D$ of $D$  and a biholomorphism $\Upsilon : U \to V$ extending the identity on $D$.
The existence of such a divisor on a complex manifold is certainly a very restrictive one, generalizing the structure 
of a product manifold with $D$ as one of the factors. In fact, since the seminal work of Grauert \cite{Gr}, it has been clear that a subtle relation exists with the property of $D$ being a {\em{splitting}} divisor in $M$ whose definition is the following

\begin{defin}
A divisor $D$ splits in $M$ if  the exact sequence 
$$0 \rightarrow TD \rightarrow TM |_D \rightarrow N \rightarrow 0$$

splits as sequence of vector bundles over $D$.
\end{defin}

\noindent In a private communication, S. Donaldson has proposed the splitting condition as the optimal one for having conical metrics
of bounded curvature. Our work is then a first step in confirming his prediction.

\noindent Clearly the Grauert property implies the splitting one, and in fact it can be considered a splitting property ``at infinite order"
as nicely explained in the work of Abate-Bracci-Tovena \cite{abt2}.

An example of divisor satisfying the Grauert property is the exceptional divisor of the blow-up of $\mathbf C^n$ at the origin. In this case $D \simeq \mathbf P^{n-1}$ and the normal bundle is the tautological line bundle $\mathcal O(-1)$, which coincides, as a complex manifold, with $\mathbf C^n$ blown-up at the origin. Since the blow-up of any smooth point in a manifold $M$ is performed by gluing a neighborhood of such $\mathbf P^{n-1}$, it follows that any exceptional divisor obtained by blowing-up a smooth point satisfies the Grauert condition.

In general, in order to get the Grauert property one must look at the cohomology groups $H^1(D, TD \otimes N^{-k})$ and $H^1(D, N^{-k})$ which are where the obstructions live \cite{Gr}.
The vanishing of the obstructions for all $k>0$ gives an isomorphism between the formal completions of $M$ and $N$ along $D$. In order to get the existence of a biholomorphism $\Upsilon$ inducing such isomorphism, one has to impose some geometrical assumptions on $D$.
In particular Grauert proved that the vanishing of $H^1(D, TD \otimes N^{-k})$ and $H^1(D, N^{-k})$ for all $k>0$ implies the Grauert property under the assumption that $D$ is exceptional (i.e. it can be blown down to a point) and the normal bundle $N$ is negative.

Negativity of the normal bundle is not a necessary condition for a divisor to have the Grauert property.
To see that consider the case when $M$ is a projective space and $D \simeq \mathbf P^{n-1}$ is a hyperplane.
The normal bundle is then $\mathcal O(1)$, which coincides, as a complex manifold, with $\mathbf P^n \setminus \{p\}$ where $p$ is any point not belonging to $D$. Thus is clear that $D$ has the Grauert property. On the other hand, Morrow-Rossi \cite{mr} proved that a divisor of a projective space having the Grauert property is in fact a hyperplane.

Finally, recalling that for a smooth complex curve in a surface are equivalent being exceptional, having negative self-intersection and having negative normal bundle, by Serre duality and standard application of Riemann-Roch Theorem one gets the following 

\begin{prop}
The following divisors have the Grauert property:
\begin{enumerate}
\item
(direct check, see above) $D$ is a hyperplane in a projective space;
\item
(direct check, see above) $D$ is the exceptional divisor of a complex manifold blown up at a smooth point;
\item
( \cite{abt2}, Section $5$) $M$ is a complex surface, $genus(D) = 0$ and $D$ has negative self-intersection;
\item
(Laufer \cite{Lau}, Chapter VI) $M$ is a complex surface, $genus(D) \geq 1$ and $D \cdot D < 4 - 4\, genus(D)$.
\end{enumerate}
\end{prop}

The list above is just as a source of nontrivial (i.e. not isomorphic to products) examples and is by far not an exhaustive list of known results. More information on the subject can be found in the work Abate-Bracci-Tovena \cite{abt2} and references therin.

\end{document}